\documentclass[reqno]{amsart}

\usepackage{amsmath,amssymb,amsthm}

\usepackage{graphicx}

\usepackage{float} 

\theoremstyle{plain}
\newtheorem{theorem}{Theorem}[section]
\newtheorem{lemma}[theorem]{Lemma}
\newtheorem{prop}[theorem]{Proposition}
\newtheorem{corollary}[theorem]{Corollary}

\newtheorem{remark}[theorem]{Remark}
\newtheorem{problem}[theorem]{Problem}

\numberwithin{equation}{section}

\usepackage{amsmath,amssymb,amsthm,mathrsfs}
\usepackage{longtable}
\usepackage{enumerate}
\allowdisplaybreaks

\begin{document}

\title{On the minimum value of the condition number of polynomials}

\author[C. Beltr\'an]{Carlos Beltr\'an}
\address{Departamento de Matem\'aticas, Estad\'{\i}stica y \\Computaci\'on,
Universidad de Cantabria. 39005. Santander, Spain}
\email{carlos.beltran@unican.es}

\author[F. Lizarte]{F\'atima Lizarte}
\address{Departamento de Matem\'aticas, Estad\'{\i}stica y \\Computaci\'on,
Universidad de Cantabria. 39005. Santander, Spain}
\email{fatima.lizarte@unican.es}


\begin{abstract}
In 1993, Shub and Smale posed the problem of finding a sequence of univariate polynomials of degree $N$ with condition number bounded above by $N$. In \cite{BEMO19} it was proved that the optimal value of the condition number is of the form $O(\sqrt{N})$, and the sequence demanded by Shub and Smale was described by a closed formula (for large enough $N\geqslant N_0$ with $N_0$ unknown) and by a search algorithm for the rest of the cases. In this paper we find concrete estimates for the constant hidden in the $O(\sqrt{N})$ term and we describe a simple formula for a sequence of polynomials whose condition number is at most $N$, valid for all $N=4M^2$, with $M$ a positive integer.
\end{abstract}


\keywords{Condition number, polynomials, spherical points}

\date{\today}

\maketitle


\section{Introduction}
\subsection{Statement of the main problem}
The condition number of a polynomial at a root is a measure for the first order variation of the root under small perturbations of the polynomial. It has different formulas and properties depending on how these changes are measured, see for example \cite{Demmel,TB}. Among the most popular and useful definitions is the one given by Shub and Smale in \cite{SS93I,SS93III}, where polynomials are first homogenized (hence the zeros lie in $\mathbb P(\mathbb{C}^2)$) and Bombieri norm is used to measure the perturbation of the polynomial. The concrete definition of Shub and Smale's {\em normalized condition number} $\mu_{\text{norm}}$ and some of its properties are recalled in a later section.

In \cite{SS93II,SS93III} it was proved that with probability at least $1/2$, (a certain choice of) random polynomials have condition number at most $N$, leading to the following:

\begin{problem}[Main Problem in \cite{SS93III}]\label{Prob_cond}
Find explicitly a family of polynomials of degree $N$ whose condition number is at most $N$.
\end{problem}
(The authors of \cite{SS93III} also relaxed the problem changing ``at most $N$'' to ``at most $N^c$ for any constant $c$, say $c=100$''.) By ``find explicitely'' they mean ``giving a handy description'' or describing a BSS algorithm --that essentially means an algorithm where exact real arithmetic is available, see \cite{BlCuShSm98}-- to solve the problem. During his plenary conference at the FoCM'14 meeting in Montevideo, Shub referred to this question as ``finding hay in the haystack'', since we know that a lot of such polynomials exist but it just turned out to be quite difficult to describe one!

The motivation of Shub and Smale was the search of a good starting polynomial to be used in homotopy methods for polynomial root finding, that is the one--dimensional case of Smale's 17th problem. Smale's 17th was finally solved without finding the solution to Problem \ref{Prob_cond} nor its high--dimensional analogous, see \cite{BP, BuCu, Lairez} or the monograph \cite{BuCubook}, leaving these questions open (see the Open Problems section in \cite{BuCubook}). Problem \ref{Prob_cond} was finally solved in \cite{BEMO19} where it was proved that:
\begin{enumerate}
  \item There exists a constant $a>0$ such that the condition number of any degree $N$ polynomial is at least $a\sqrt{N}$.
  \item There exist an explicit construction of a polynomial of any degree, given by its zeros, and a constant $b>0$ such that the condition number of the $N$--th degree polynomial is at most $b\sqrt{N}$.
  \end{enumerate}
   From \cite{Ujue} we have the concrete value $a\geqslant e^{C_{\log}}/2$ where $C_{\log}$ is defined by \eqref{eq:Clog} and bounded in \eqref{eq:Clogbound}, but the value of $b$ is not known. As a consequence, one gets an algorithm to generate a degree $N$ polynomial whose condition number is at most $N$: run in parallel a search algorithm (based on enumeration of rational zeros) and the sequence of item (2). Since computing the condition number given the zeros is immediate, and since the polynomials in the sequence of (2) eventually have a condition number smaller than $N$, this produces a polynomial time algorithm for Problem \ref{Prob_cond}.

The solution of \cite{BEMO19} is thus an algorithm to generate the demanded sequence, and it certainly solves Problem \ref{Prob_cond}, but it leaves an open question behind:
\begin{problem}[Main Problem after \cite{BEMO19}]\label{Prob_cond2}
Find an explicit formula for a family of polynomials of degree $N$ whose condition number is at most $N$. Also, find asymptotic bounds for the minimum condition number of a degree $N$ polynomial, $N\to\infty$.
\end{problem}
\subsection{Main result}
In these pages we make some partial progress in Problem \ref{Prob_cond2}. More exactly, we prove for the first part of this problem:
\begin{theorem}[Main result]\label{th:mainintro}
Let $N=4M^2,$ with $M\geqslant 1$ a positive integer. Define
\begin{equation*}
r_j=4j, \quad h_j=1-\frac{4j^2}{N},
\end{equation*}
for $1\leqslant j \leqslant M$ and consider the polynomial of degree $N$ given by
\begin{equation*}
P_N(z)=\displaystyle(z^{r_M}-1)\prod_{j=1}^{M-1}(z^{r_j}-\rho(h_j)^{r_j})(z^{r_j}-\rho(h_j)^{-r_j}),
\end{equation*}
where $\rho(x)=\sqrt{\frac{1+x}{1-x}}$. Then $\mu_{\text{norm}}(P_N)\leqslant \min(N,(19/2)\sqrt{N+1})$.
\end{theorem}
The first three polynomials of our sequence can be seen in Table \ref{table:tabla_poly}.

\begin{table}[htbp]
\begin{center}
\begin{tabular}{|l|l|l|}
\hline
$M$ & $N$ & Polynomial  \\
\hline \hline
1 & 4  &	$z^4-1$		\\ \hline
2 & 16 &	$(z^8-1)(z^4-49)(z^4-1/49)$		\\ \hline
3 & 36 &	$(z^{12}-1)(z^8-2401/16)(z^8-16/2401)(z^4-289)(z^4-1/289)$ \\ \hline
\end{tabular}
\caption{The first three polynomials of the sequence constructed in Theorem \ref{th:mainintro} corresponding to degrees 4, 16 and 36, respectively.}
\label{table:tabla_poly}
\end{center}
\end{table}

The zeros of above polynomial $P_N$ correspond, under the stereographic projection, to the spherical points of a set $\mathcal{P}_N$ described in Section \ref{Sect:PN}.
Modifying $\mathcal P_N$ slightly, one can very likely adapt our proof to similar subsequences such as, say, $N=4M^2+1$ or $N=4M^2+2M$, but solving the problem for general $N$ is still out of reach, since the explicit computations become too complicated.

Our method of proof also produces the upper bound in the following corollary, which is a first answer to the second part of Problem \ref{Prob_cond2} (the lower bound is proved by U. Etayo in \cite{Ujue} and uses a recent result by S. Steinerberger \cite{Stefan}):
\begin{corollary}\label{cor:limsinf}
The minimum condition number $\alpha_N=\inf\{\mu_{\text norm}(P):deg(P)=N\}$ of a degree $N$ polynomial satisfies
\[
0.454\ldots\leqslant \frac{e^{C_{\log}}}2\leqslant \liminf_{N\to\infty}\frac{\alpha_N}{\sqrt{N}}\leqslant \frac{\sqrt{3}}{2}e^{15/8}=5.647\ldots
\]
\end{corollary}

A related problem of great importance is that of finding the optimal constant $C_N$ in the multiterm Bombieri inequality
\[
\prod_{i=1}^N{\|z-z_i\|}\leqslant C_N \left\|\prod_{i=1}^N(z-z_i)\right\|,
\]
which compares the Bombieri-Weyl norm of a polynomial and the product of the norms of its linear factors. Finding the optimal value of $C_N$ is a formidable challenge! A recent breackthrough by U. Etayo \cite{Ujue} is that the optimal value of this constant is essentially $\sqrt{e^N/(N+1)}$. More precisely, define $K_N$ by
\[
C_N=K_N\sqrt{\frac{e^N}{N+1}}.
\]
Then, we have $\alpha\leqslant K_N\leqslant 1$ for some $\alpha>0$ which is independent of $N$. No lower bounds on $\alpha$ are known till now, but from \cite[Th. 4.5]{Ujue} and our main result above we deduce:
\begin{equation}\label{eq:KN}
\limsup_{N\to\infty}K_N\geqslant \frac{e^{C_{\log}}}{\sqrt{3}e^{15/8}}\geqslant 0.08.
\end{equation}

\subsection{Relation to well--distributed spherical points and the logarithmic energy}
We will define our sequence of polynomials by its zeros, which are in turn seen as points in the unit $2$--sphere $\mathbb{S}$ via the stereographic projection. It was noted in \cite{SS93III} that if a collection of spherical points $ p_1,\ldots,  p_N\in\mathbb{S}$ is very well distributed in the sense that it quasi--minimizes the logarithmic energy
\begin{equation}\label{Logarithmic_energy}
\mathcal{E}( p_1,\ldots,  p_N)=\displaystyle\sum_{i \neq j} \log\frac{1}{| p_i- p_j|},
\end{equation}
then the associated complex points are the zeros of a well--conditioned polynomial. More precisely, let us denote by $m_N$ the minimum possible value of the logarithmic energy,
$$
m_N=\min_{ p_1,\ldots,  p_N\in\mathbb{S}}\mathcal{E}( p_1,\ldots,  p_N).
$$
The main result of \cite{SS93III} is that if $\mathcal{E}( p_1,\ldots,  p_N)\leqslant m_N+c\log N$ then the condition number of the corresponding polynomial is at most $\sqrt{N^{1+c}(N+1)}$, thus solving the relaxation of Problem \ref{Prob_cond}. (Note that our notation is slightly different from that of \cite{SS93III}: we use the unit sphere instead of the Riemann sphere and our definition of the log--energy ranges over $i\neq j$ instead of $i<j$. Our notation is the most frequent nowadays).

Inspired by this result, Shub and Smale posed the problem of finding collections of spherical points with quasioptimal $\log$--energy. This later was included in Smale's famous list of Problems for the XXI century  \cite{Smale2000}:
\begin{problem}[Smale's 7th Problem]\label{Prob7_Smale}
Can one find $ p_1,\ldots,  p_N\in\mathbb{S}$ such that $\mathcal{E}( p_1,\ldots,  p_N) \leqslant  m_N+c\log N$ for some universal constant $c$?
\end{problem}
The value of $m_N$ is not still well known. After \cite{Wagner89,RSZ94,Dubickas96,Brauchart08,BS18}, we have
\begin{equation}\label{eq:Clog}
m_N=\kappa N^2-\frac{1}{2}N\textrm{log}N+C_{\log}N+o(N),
\end{equation}
where $C_{\log}$ is a constant and, denoting by $d\sigma$ the normalized uniform measure in $\mathbb S$,
\begin{equation}\label{eq:continuousenergy}
\kappa=\int_{x,y\in\mathbb{S}}\log\frac{1}{|x-y|}d\sigma(x)d\sigma(y)=\frac{1}{2}-\log 2<0,
\end{equation}
is the continuous energy. From \cite{Stefan} and \cite{BS18}, we have that
\begin{equation}\label{eq:Clogbound}
-0.0954\ldots \leqslant C_{\log} \leqslant 2\log2 +\frac{1}{2}\log\frac{2}{3}+3\log\frac{\sqrt{\pi}}{\Gamma(1/3)}=-0.0556\ldots
\end{equation}
The upper bound has been conjectured to be an equality, see \cite{BHS12,BS18} and the monograph \cite{BHS19} for context.
We stress that our construction of the point set for Theorem \ref{th:mainintro} does not solve Problem \ref{Prob7_Smale}: its log--energy is of the form $\kappa N^2-\frac12N\log N+O(N)$ (this can be deduced directly from \eqref{Logarithmic_energy} and Corollary \ref{cor:D} or seen as a consequence of \cite[Th. 1.5]{Ujue}).


\subsection{Condition number of polynomials}

We now give the precise definition and some properties of the condition number of polynomials. Let us consider a bivariate homogeneous polynomial with complex coefficients of degree $N\geqslant 1$,
\begin{equation*}
h(x,y)=\displaystyle\sum_{i=0}^N a_ix^iy^{N-i}, \quad a_i\in\mathbb{C}, \,\, a_N\neq 0.
\end{equation*}
The zeros of $h$ lie in the complex projective space $\mathbb{P}(\mathbb{C}^2)$.
Following \cite{SS93I}, the normalized condition number of $h$ at a zero $\zeta\in\mathbb{P}(\mathbb{C}^2)$ is
\begin{equation*}
\mu_{\text{norm}}(h,\zeta)=
\left\{
\begin{array}{ll}
N^{1/2}\|(Dh(\zeta)|_{\zeta^{\perp}})^{-1}\|\|h\|\|\zeta\|^{N-1}, \quad &\text{if }\, \exists (Dh(\zeta)|_{\zeta^{\perp}})^{-1},\\
+\infty, \quad &\text{otherwise}.
\end{array}
\right.
\end{equation*}
Here, $Dh(\zeta)|_{\zeta^{\perp}}$ is the restriction of the derivative $Dh(\zeta)=\left(\frac{\partial}{\partial x}h\quad \frac{\partial}{\partial y}h\right)_{(x,y)=\zeta}$ to the orthogonal complement of  $\zeta$ in $\mathbb{C}^2$, and $\|h\|$ is the Bombieri-Weyl norm (also known as Kostlan or Bombieri or Weyl norm) of $h$, defined as
\begin{equation*}
\|h\|= \left(\displaystyle\sum_{i=0}^N {N \choose i}^{-1} |a_i|^2\right)^{1/2}.
\end{equation*}
If $\zeta$ is a double root of $h$, then by definition $\mu_{\text{norm}}(h,\zeta)=\infty$.
On the other hand, if there is not mention to a concrete root of $h$, then we define
\begin{equation*}
\mu_{\text{norm}}(h)=\max_{\zeta\in \mathbb{P}(\mathbb{C}^2):h(\zeta)=0}\mu_{\text{norm}}(h,\zeta).
\end{equation*}
Let
\begin{equation*}
f(z)=\displaystyle\sum_{i=0}^N a_iz^i, \quad a_N\neq 0,
\end{equation*}
be an univariate polynomial of degree $N$ with complex coefficients and $z\in\mathbb{C}$ a zero of
$f$. Consider the homogeneous counterpart of $f$,
\begin{equation*}
h(x,y)=\displaystyle\sum_{i=0}^N a_ix^iy^{N-i},
\end{equation*}
and define
\begin{equation*}
\mu_{\text{norm}}(f,z)=\mu_{\text{norm}}(h,(z,1)), \quad \mu_{\text{norm}}(f)=\max_{z\in\mathbb{C}:f(z)=0}\mu_{\text{norm}}(f,z).
\end{equation*}
Taking $\|f\|=\|h\|$ and expanding the derivative, it turns out that
\begin{equation}\label{eq:mu}
\mu_{\text{norm}}(f,z)=\frac{N^{1/2}(1+|z|^2)^{\frac{N-2}{2}}}{|f'(z)|}\|f\|,
\end{equation}
which allows us to easily compute the condition number for simple cases (see \cite{Beltran15} for an elementary proof of this last formula).


\subsection{An alternative formula for the condition number and idea of our proof}
Since a polynomial is (up to a multiplicative constant) defined by its zeros and these can be seen as spherical points, one can aim to give a formula for the condition number of a polynomial that depends uniquely on the associated spherical points. Shub and Smale accomplished this task. Adapting the notation of \cite{SS93III} to ours, we have:
\begin{prop}
Let $P(z)=\prod_{i=1}^N(z-z_i)$ be a polynomial and denote by $ p_i$ the point in $\mathbb{S}$
obtained from the inverse stereographic projection of each $z_i$. Then the condition number of $P$ equals
\begin{equation}\label{exp_cond_S2}
\mu_{\text{norm}}(P)=\frac{1}{2}\sqrt{N(N+1)}\max_{1 \leqslant i \leqslant N}\frac{\left(\int_{\mathbb{S}} \prod_{j=1}^N |p- p_j|^2 d\sigma(p)\right)^{1/2}}{\prod_{j\neq i}| p_i- p_j|}.
\end{equation}
\end{prop}
As in \cite{BEMO19}, we will start from a geometrical construction of a point set $\mathcal P_N$ (see Figure \ref{fig:construction} for a graphical description). Its main features are:
\begin{enumerate}
\item The $N$ spherical points are distributed in $2M-1$ parallels of varying height in the sphere, with the $M$--th parallel being the equator.
\item The parallel at height $h_j$ contains $r_j$ points which are (up to a homotety and a traslation) a set of $r_j$ roots of the unity. They may have a phase or not, this is not important for our proof.
\item The values of $h_j$ are chosen in such a way that there is a {\em band} of relative area $r_j/N$ whose central height is $h_j$.
\item The construction is equatorially symmetric: $h_j=-h_{2M-j}$ and $r_j=r_{2M-j}$.
\end{enumerate}
 Once we have defined our set of points, in order to prove our main result we proceed as follows:
 \begin{enumerate}
\item[(A)] Given any $q\in\mathbb S$ and any band $B$ in the sphere, we consider the central parallel $Q$ of $B$, and we compare the integral $I_B$ of $\log|p-q|$ when $p$ lies in the band with the expected value $\Tilde{I}_{Q}$ of the same function when $p$ lies in $Q$. We conclude that $I_B\approx \nu(B) \Tilde{I}_Q$ where $\nu(B)$ is the normalized area of $B$.
\item[(B)] Given any $q\in\mathbb S$, we divide the integral $I=-\kappa$ of $\log|p-q|$ with $p\in\mathbb S$ in the different bands associated to our point set. From (A), the value $I_j$ in each band $B_j$ is similar to $\nu(B_j)\Tilde{I}_j $ where $\Tilde{I}_j $ is the expected value of the function along the parallel $Q_j$, that is $I_j\approx (r_j/N)\Tilde{I}_j $ and $-\kappa N=I N=N\sum_j I_j\approx \sum_j r_j\Tilde I_j$. The difference between $-\kappa N$ and $\sum_j r_j\Tilde{I}_j$ turns out to be rather small except if $q$ is too close to the poles. 
\item[(C)] We then compare the value of $r_j \Tilde{I}_j $ with that of $\sum_{k=0}^{r_j-1}\log|q-p_{j,k}|$ where the $p_{j,k}$ are the $r_j$ points in the corresponding parallel. Both quantities are again very similar (except for the parallel which is the one closest to $q$, where the discrete sum can diverge to $-\infty$). From this, we get
    \[
    \prod_{i=1}^n|q-p_i|=e^{\sum_{j=1}^{2M-1}\sum_{k=0}^{r_j-1}\log|q-p_{j,k}|}\lesssim e^{\sum_{j=1}^{2M-1}r_j \Tilde{I}_j}\approx e^{-\kappa N}.
    \]
     This essentially gives an upper bound for the numerator in \eqref{exp_cond_S2}, once the details are settled.
\item[(D)] The same kind of argument, using also that our point set is well--separated, produces a lower bound $\prod_{j\neq i}|p_i-p_j|\gtrsim \sqrt{N}e^{-N\kappa}$, valid for all fixed $i$, for the denominator of \eqref{exp_cond_S2}. This almost finishes the proof of our main result.
 \end{enumerate}
 This procedure is similar to that of \cite{BEMO19}, but in this paper all the appearances of $\approx, \lesssim,\gtrsim$ are estimated with concrete constants. One benefit that we get is that our point set is more simple, since in \cite{BEMO19} the points need to be distributed in the parallels of height $h_j$ but also a part of them is sent to the parallels delimiting the spherical bands, while we only need to allocate points in the central parallels.
 \begin{remark}
  The construction in \cite{BEMO19} has a property that ours does not: the associated discrete measure can be used to approximate the continuous integral of $\log|p-q|$ up to a constant order for any fixed $q$. This property (which is the reason to send part of the points to the parallels delimiting the bands) is a key point in the proof of the main result of \cite{BEMO19}. Our construction only gets this if $q$ is not too close to the north and south poles, yet we are able to prove our main theorem from this weaker property.
 \end{remark}


\subsection*{Organization of the paper}
In the next section, we state the construction of the set of spherical points which will be, under the stereographic projection, the zeros of the sequence of well--conditioned polynomials. In Section \ref{sec:comparison} we prove (A); in Section \ref{sec:Proof1} we prove (B); sections \ref{sec:numerador} and \ref{sec:denominador} are devoted to proving (C) and (D) respectively. Finally, the main result is proved in Section \ref{sec:final}.


\section{Geometrical description of the set of points $\mathcal{P}_{N}$ in $\mathbb{S}$}\label{Sect:PN}

We now construct our set of points $\mathcal{P}_N=\{p_1,\ldots,p_N\}$ in $\mathbb{S}$. We denote by $Q_h$ the parallel of height $h$,
\begin{equation*}
Q_h=\{(x,y,z)\in\mathbb{S} : z=h\}, \quad -1\leqslant h \leqslant 1.
\end{equation*}

Let $M$ be a positive integer, define $N=4M^2$, and let
\begin{equation*}
r_j=\begin{cases}4j, &1\leqslant j\leqslant M,\\4(2M-j), &M\leqslant j\leqslant 2M-1,\end{cases}
\quad
h_j=\begin{cases}1-\frac{j^2}{M^2}, &1\leqslant j\leqslant M, \\-1+\frac{(2M-j)^2}{M^2}, &M\leqslant j\leqslant 2M-1.\end{cases}
\end{equation*}
Note that $N=r_1+\cdots+r_{2M-1}$ and the claimed symmetric properties $r_j=r_{2M-j}$, $h_j=-h_{2M-j}$, and note also that $h_M=0$. Our point set is constructed by taking $r_j$ equally spaced points in each of the parallels $Q_j=Q_{h_j}$. We will refer to $Q_j$ just as the $j$--th parallel.
For all $1 \leqslant j \leqslant 2M-1$, we define the $j$--th band as
\begin{equation*}
B_j=\{(x,y,z)\in\mathbb{S}, \,H_j \leqslant z \leqslant H_{j-1}\},
\end{equation*}
being
\begin{equation}\label{frontera_banda}
H_j=\begin{cases}1-\frac{j(j+1)}{M^2}, & 0 \leqslant j \leqslant M-1,\\-1+\frac{(2M-j-1)(2M-j)}{M^2}, &M\leqslant j\leqslant 2M-1.\end{cases}
\end{equation}
Observe that $Q_j$ is the central parallel of the band $B_j$, in the sense that
\begin{equation*}
h_j=\frac{H_{j-1}+H_j}{2}=H_{j-1}-\frac{r_j}{N}=H_j+\frac{r_j}{N},\quad 1\leqslant j \leqslant 2M-1.
\end{equation*}
Moreover, note that $B_1$ and $B_{2M-1}$ are just two spherical caps surrounding the north and the south pole respectively and that  $\mathbb{S}= \displaystyle\cup_{j=1}^{2M-1}B_{j}$. The relative area of each band is (use for example Lemma \ref{int_S2}):
\begin{equation*}
\nu(B_j)=\frac{H_{j-1}-H_j}{2}=\frac{r_j}{N}, \quad 1\leqslant j \leqslant 2M-1.
\end{equation*}
See Figure \ref{fig:construction} for a graphical illustration of this construction.
\begin{figure}
  \centering
  \includegraphics[width=\linewidth]{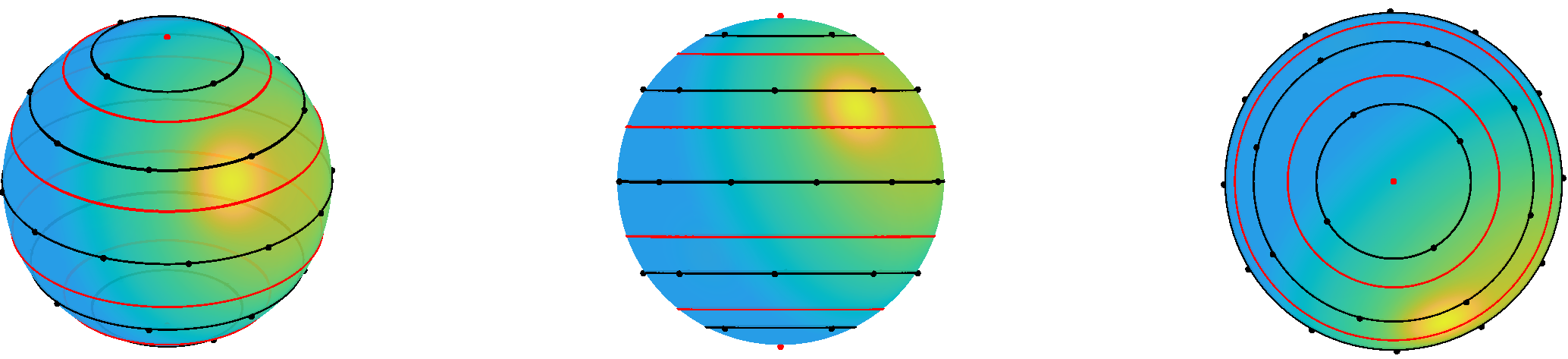}
  \caption{Our construction of spherical points for $M=3$, that is $N=36$ points, from three different points of view (left: tilted; center: equatorial; right: north pole). The parallels $Q_j$ are the black circles, and the points are equidistributed among them. The red lines are the parallels $Q_{H_j}$ which delimit the bands (the north and south poles are marked with red dots and correspond to $H_0$ and $H_{2M-1}$, but they do not belong in $\mathcal P_N$.)}\label{fig:construction}
\end{figure}


\section{Comparison of the integrals in parallels and bands}\label{sec:comparison}
This section provides a comparison tool which is independent of the construction above. It will later be applied to each of the bands $B_j$ described in Section \ref{Sect:PN}.
Let $B$ be the band contained between parallels of height $h-\epsilon$ and $h+\epsilon$, with $Q=Q_h$ the central parallel. The relative area of $B$ is $\epsilon$.
We now show that, for any {\em fixed} $q=(a,b,c)\in\mathbb S$, the integral $I_B(q)$ of $\log|p-q|$ with $p$ chosen in $B$ is approximately equal to the relative area of the band, times the expected value $\Tilde I_Q(q)$ of the same function in the central parallel. From \cite[Prop. 2.2]{Diamond}, for the parallel of height $t\in[-1,1]$ the expected value $\Tilde I_{Q_t}(q)$ satisfies

\begin{equation}\label{Valor_fp(h)}
\Tilde I_{Q_t}(q)=
\left\{\begin{array}{ll}
\frac{1}{2}(\log(1+t)+\log(1-c)), \quad &\mathrm{if }\,\, t \geqslant c,\\
&\\
\frac{1}{2}(\log(1-t)+\log(1+c)), \quad &\mathrm{if }\,\, t <c.
\end{array}\right.
\end{equation}
From Lemma \ref{int_S2} we have
\begin{equation}\label{rel1}
\frac{1}{\varepsilon}I_B(q)=\frac{1}{\varepsilon}\int_{p\in B}\log|q-p|d\sigma(p)=\frac{1}{2\varepsilon}\int_{h-\varepsilon}^{h+\varepsilon}\Tilde I_{Q_t}(q)dt.
\end{equation}
\begin{lemma}[Comparison when $q$ is outside of the band]\label{lem:nuevocotaptomedio}
With the notations above, if $c\geqslant h+\varepsilon$ we have
\begin{align}\label{eq:nueva}
 0\leqslant \Tilde I_{Q}(q)-\frac{1}{\varepsilon}I_B(q)
-\frac{\varepsilon^2}{12(1-h)^2}&=\sum_{n=2}^{\infty}
\frac{\varepsilon^{2n}}{4n(2n+1)(1-h)^{2n}}\\
\nonumber & \leqslant \frac{1}{2}\left(\frac{5}{6}-\log2\right)\frac{\varepsilon^{4}}{(1-h)^{4}},
\end{align}
and if $c\leqslant h-\varepsilon$ we have
\begin{align}\label{eq:nueva2}
0\leqslant \Tilde I_{Q}(q)-\frac{1}{\varepsilon}I_B(q)
-\frac{\varepsilon^2}{12(1+h)^2}&=\sum_{n=2}^{\infty}
\frac{\varepsilon^{2n}}{4n(2n+1)(1+h)^{2n}}\\
\nonumber & \leqslant \frac{1}{2}\left(\frac{5}{6}-\log2\right)\frac{\varepsilon^{4}}{(1+h)^{4}}.
\end{align}
\end{lemma}

\begin{proof}
Assume first that $c\geqslant h+\varepsilon$. Let $U$ be the right hand term in \eqref{eq:nueva}. Then,
\begin{align*}
U=&\frac12\log(1-h)-\frac{1}{4\varepsilon}\int_{-\varepsilon}^\varepsilon \log(1-h+t)\,dt-\frac{\varepsilon^2}{12(1-h)^2}\\
=&-\frac{1}{4\varepsilon}
\int_{-\varepsilon}^{\varepsilon}\log\left(1+\frac{t}{1-h}\right)\,dt
-\frac{\varepsilon^2}{12(1-h)^2}\\
=&-\sum_{n=1}^\infty\frac{1}{4\varepsilon}\int_{-\varepsilon}^\varepsilon
\frac{(-1)^{n+1}t^n}{n(1-h)^n}\,dt-\frac{\varepsilon^2}{12(1-h)^2}.
\end{align*}
The terms with odd $n$ in the sum integrate to $0$ and hence we have
\begin{align*}
U=&\sum_{n=1}^\infty
\frac{\varepsilon^{2n}}{4n(2n+1)(1-h)^{2n}}-\frac{\varepsilon^2}{12(1-h)^2}
=\sum_{n=2}^\infty
\frac{\varepsilon^{2n}}{4n(2n+1)(1-h)^{2n}},
\end{align*}
as wanted. The final inequality follows from noting that $\varepsilon/(1-h)\leqslant 1$ and computing the sum (\cite[(0.234.8)]{GR15}). The other case ($c\leqslant h-\varepsilon$) is proved the same way.
\end{proof}

\begin{lemma}[Comparison when $q$ is inside of the band]\label{comp_paral_band_p}
With the notations above, assume now that $h-\varepsilon \leqslant c \leqslant h+\varepsilon.$ Then,
\begin{equation}
-\frac{\varepsilon/4}{1-h^2}\leqslant  \Tilde I_{Q}(q)-\frac{1}{\varepsilon}I_{B}(q)\leqslant
\begin{cases}\frac{(1-\log 2)\varepsilon^2}{2(1-h)^2},&h\leqslant c\leqslant h+\varepsilon,\\
\frac{(1-\log 2)\varepsilon^2}{2(1+h)^2},&h-\varepsilon\leqslant c\leqslant h.\end{cases}
\end{equation}
\end{lemma}

\begin{proof}
Assume first that $c\in[h,h+\varepsilon]$, and define
\begin{align*}
U(c)=&\Tilde I_{Q}(q)-\frac{1}{\varepsilon}
I_{B}(q)\\
=&\frac{1}{2}\log(1-h)+\frac12\log(1+c)-\frac{1}{4\varepsilon}\int_{h-\varepsilon}^c\log(1-t)+\log(1+c)\,dt
\\&-\frac{1}{4\varepsilon}\int_c^{h+\varepsilon}\log(1+t)+\log(1-c)\,dt.
\end{align*}
With some little arithmetic it is easy to see that
\[
U'(c)=\frac{h+\varepsilon-c}{\varepsilon(2-2c^2)}\geqslant 0.
\]
The maximum and the minimum of $U$ in the interval $c\in[h,h+\varepsilon]$ are thus in the extremes. The case $c=h+\varepsilon$ is covered by \eqref{eq:nueva}, which (using $\varepsilon/(1-h)\leqslant 1$) yields:
\begin{equation*}
U(c)\leqslant U(h+\varepsilon)=\sum_{n=1}^\infty
\frac{\varepsilon^{2n}}{4n(2n+1)(1-h)^{2n}}\leqslant \frac{\varepsilon^2}{(1-h)^2}\frac14\sum_{n=1}^\infty
\frac{1}{n(2n+1)},
\end{equation*}
and again from \cite[(0.234.8)]{GR15} we obtain the result.
For the minimum, we have
\begin{align*}
U(c)\geqslant U(h)=&\frac14(\log(1-h)+\log(1+h))-\frac{1}{4\varepsilon}\left(\int_{h-\varepsilon}^h\log(1-t)\,dt+\int_h^{h+\varepsilon}\log(1+t)\,dt\right)\\
=&-\frac{1}{4\varepsilon}\left(\int_{h-\varepsilon}^h\log\left(1+\frac{h-t}{1-h}\right)\,dt+\int_h^{h+\varepsilon}\log\left(1+\frac{t-h}{1+h}\right)\,dt\right)
\\
=&-\frac{1}{4\varepsilon}\left(\int_0^{\varepsilon}\log\left(1+\frac{t}{1-h}\right)\,dt+\int_0^{\varepsilon}\log\left(1+\frac{t}{1+h}\right)\,dt\right).
\end{align*}
Expanding the logarithms in power series and integrating termwise we get
\begin{align*}
U(c)\geqslant U(h)=&\frac{1}{4}\sum_{n= 1}^\infty\frac{(-1)^n\varepsilon^n}{n(n+1)}\frac{1}{(1-h)^n}+\frac{1}{4}\sum_{n= 1}^\infty\frac{(-1)^n\varepsilon^n}{n(n+1)}\frac{1}{(1+h)^n},
\end{align*}
and since $\varepsilon/(1-h)$ and $\varepsilon/(1+h)$ are both at most equal to $1$ we see that both alternating series have decreasing terms, thus concluding:
\[
U(c)\geqslant U(h)\geqslant-\frac{\epsilon}{8(1-h)}-\frac{\epsilon}{8(1+h)}=-\frac{\epsilon/4}{1-h^2},
\]
and the lemma follows. The other case ($c\in[h-\varepsilon,h]$) is done the same way.
\end{proof}


\section{Comparison between $-\kappa N$ and $\sum_j r_j\Tilde{I}_j $}\label{sec:Proof1}
 Recall that:
\begin{itemize}
  \item $I_j(q)$ is the integral of $\log|p-q|$ when $p$ lies in the $j$--th band $B_j$, so that
\begin{equation}\label{eq:kap}
-\kappa=\int_{\mathbb S}\log|p-q|\,d\sigma(p)=\sum_{j=1}^{2M-1}I_j(q),\quad q\in\mathbb S.
\end{equation}
  \item $\Tilde{I}_j (q)$ is the expected value of $\log|p-q|$ when $p$ lies in the $j$--th parallel $Q_j$.  From \eqref{Valor_fp(h)} we have
      \begin{equation}\label{Valor_Itildej}
\Tilde{I}_j (q)=
\left\{\begin{array}{ll}
\frac{1}{2}(\log(1+h_j)+\log(1-c)), \quad &\mathrm{if }\,\, h_j \geqslant c,\\
&\\
\frac{1}{2}(\log(1-h_j)+\log(1+c)), \quad &\mathrm{if }\,\, h_j <c.
\end{array}\right.
\end{equation}
\end{itemize}
The results in Section \ref{sec:comparison} yield $N I_j(q)\approx r_j\Tilde{I}_j(q)$ where the meaning of $\approx$ is precise and has different bounds for $q$ outside and inside of the band $B_j$.

\begin{lemma}\label{const_cond}
Let  $q=(a,b,c)\in B_\ell\subseteq \mathbb{S}$ with $\ell\leqslant M$ and $M\geqslant 5$. Let
\begin{equation}\label{Sum_SN}
S_N=S_N(q)=\displaystyle\sum_{j=1}^{2M-1} r_j\Tilde{I}_j (q).
\end{equation}
Then,
\begin{equation*}
-1\leqslant S_N +  N\kappa-T(\ell)\leqslant \frac{2(1-\log 2)}{\ell}+\frac{1}{15},
\end{equation*}
where
\begin{equation}\label{eq:T}
T(\ell)=\sum_{j=1}^{\ell-1}\frac{r_j^3}{12N^2(1+h_j)^2}+\sum_{j=\ell+1}^{2M-1}\frac{r_j^3}{12N^2(1-h_j)^2}.
\end{equation}
\end{lemma}
\begin{proof}
Let
\begin{align*}
(\star)=&\,S_N + N\kappa-T(\ell)\\
\stackrel{\text{\eqref{eq:kap}}}
=&\,r_\ell \left(\Tilde{I}_\ell(q)-\frac{N}{r_\ell}I_\ell(q)\right)\\
&+\sum_{j=1}^{\ell-1}r_j\left(\Tilde{I}_j(q)-\frac{N}{r_j}I_j(q)-\frac{r_j^2}{12N^2(1+h_j)^2}\right)\\
&+\sum_{j=\ell+1}^{2M-1}r_j\left(\Tilde{I}_j(q)-\frac{N}{r_j}I_j(q)-\frac{r_j^2}{12N^2(1-h_j)^2}\right).
\end{align*}
From Lemmas \ref{lem:nuevocotaptomedio} (with $\varepsilon=r_j/N$) and \ref{comp_paral_band_p} (with $\varepsilon=r_\ell/N$) we have on the one hand
\[
(\star)\geqslant -\frac{ r_\ell^2}{4N(1-h_\ell^2)}=-\frac{1}{2-\frac{\ell^2}{M^2}}\geqslant -1,
\]
and on the other hand
\begin{align*}
(\star)\leqslant & \frac{(1-\log2)\,r_\ell^3}{2N^2(1-h_\ell)^2}+\frac{1}{2}\left(\frac{5}{6}-\log2\right)\left(\sum_{j=1}^{\ell-1}\frac{r_j^{5}}{N^4(1+h_j)^{4}}+
\sum_{j=\ell+1}^{2M-1}\frac{r_j^{5}}{N^4(1-h_j)^{4}}\right)\\
\leqslant &\frac{2(1-\log 2)}{\ell}+\frac{1}{2}\left(\frac{5}{6}-\log2\right)\left(2\sum_{j=1}^{M-1}\frac{r_j^{5}}{N^4(1+h_j)^{4}}+\sum_{j=\ell+1}^{M}\frac{r_j^{5}}{N^4(1-h_j)^{4}}\right)\\
\leqslant &\frac{2(1-\log 2)}{\ell}+\frac{1}{2}\left(\frac{5}{6}-\log2\right)\left(8\sum_{j=1}^{M-1}\frac{j^{5}}{(2M^2-j^2)^{4}}+4\sum_{j=\ell+1}^{M}\frac{1}{j^3}\right)\\
\leqslant &\frac{2(1-\log 2)}{\ell}+\frac{1}{2}\left(\frac{5}{6}-\log2\right)\left(8 \displaystyle\sum_{j=1}^{M-1} \frac{j^5}{M^8} + 4\displaystyle\sum_{j=2}^{\infty}\frac{1}{j^3}\right)\\
\leqslant &\frac{2(1-\log 2)}{\ell}+\frac{1}{2}\left(\frac{5}{6}-\log2\right)\left(\frac{1}{30}+4[\zeta(3)-1]\right),
\end{align*}
where $\zeta(3)$ denotes Ap\'ery's constant. Note that we have used \cite[(0.121.5)]{GR15} and $M\geqslant 5$ to deduce
$$
8\displaystyle\sum_{j=1}^{M-1} \frac{j^5}{M^8}=\frac{2(M-1)^2}{3M^5}\left(2(M-1)-\frac{1}{M}\right)\leqslant \frac1{30}.
$$
The lemma follows after some arithmetic bounding $\log 2\geqslant 0.69$ and $\zeta(3)\leqslant 1.203$.
\end{proof}
The following result offers us a lower and upper bound for $T(\ell)$.

\begin{lemma}\label{lem:aux1}
Let $q\in B_{\ell}$ with $\ell\leqslant M$. Then,
\begin{equation*}
\frac{1}{3}\log\frac{M+1}{\ell+1}\leqslant T(\ell) \leqslant \frac{1}{3}\log\frac{M}{\ell}+\frac{1}{6}.
\end{equation*}
\end{lemma}

\begin{proof}
From Definition \eqref{eq:T} and symmetry with respect to the equator, we have
\begin{align*}
T(\ell)=& \frac{1}{12}\displaystyle\sum_{j=1}^{\ell-1}\frac{r_j^3}{N^2(1+h_j)^2}+\frac{1}{12}\displaystyle\sum_{j=\ell+1}^M \frac{r_j^3}{N^2(1-h_j)^2}+\frac{1}{12}\displaystyle\sum_{j=1}^{M-1}\frac{r_j^3}{N^2(1+h_j)^2}\\
=& \frac{1}{3}\displaystyle\sum_{j=1}^{\ell-1}\frac{j^3}{(2M^2-j^2)^2}+\dfrac{1}{3}\displaystyle\sum_{j={\ell+1}}^M \frac{1}{j}+\frac{1}{3}\displaystyle\sum_{j=1}^{M-1}\frac{j^3}{(2M^2-j^2)^2}.
\end{align*}
Then,
\begin{equation*}
\frac{1}{3}\displaystyle\sum_{j=\ell+1}^M\frac{1}{j}\leqslant T(\ell) \leqslant
\frac{1}{3}\displaystyle\sum_{j=\ell+1}^M \frac{1}{j} +\frac{2}{3}\displaystyle\sum_{j=1}^{M-1}\frac{j^3}{(2M^2-j^2)^2}.
\end{equation*}
Note that $\sum_{j={\ell+1}}^M \frac{1}{j}$ vanishes for $\ell=M$.
The result follows from the following two bounds: on the one hand
\begin{equation*}
\frac{2}{3}\displaystyle\sum_{j=1}^{M-1}\frac{j^3}{(2M^2-j^2)^2} \leqslant \frac{2}{3}\displaystyle\sum_{j=1}^{M-1}\frac{j^3}{M^4}=\frac{(M-1)^2}{6M^2} \leqslant \frac{1}{6},
\end{equation*}
and on the other hand, from Lemma \ref{lem:sumalog} we have
\[
\frac{1}{3}\log\frac{M+1}{\ell+1}<\frac{1}{3}\displaystyle\sum_{j=\ell+1}^M \frac{1}{j}<\frac{1}{3}\log\frac{M}{\ell},
\]
and the proof is concluded.
\end{proof}

\begin{prop}\label{Cociente_Integral_cond_previo}
Let $q\in B_{\ell}$ with $\ell\leqslant M$ and $M \geqslant 5$. Then,
\begin{equation*}
-1\leqslant -1+\frac13\log \frac{M+1}{\ell+1}\leqslant S_N+N\kappa \leqslant \frac13\log \frac{M}{\ell}+\frac{2(1-\log 2)}{\ell}+\frac{1}{4}.
\end{equation*}
In other words, $-N\kappa\approx \sum_j r_j\Tilde I_j(q)$ where the symbol $\approx$ hides essentially $\frac13\log\frac M\ell$.
\end{prop}

\begin{proof}
Immediate from lemmas \ref{const_cond} and \ref{lem:aux1}.
\end{proof}


\section{The numerator of the condition number formula}\label{sec:numerador}
This section is devoted to get a upper bound for the term
\begin{equation*}
\log\displaystyle\prod_{i=1}^N |p-p_i| - S_N,
\end{equation*}
with $S_N=S_N(p)$ the discrete sum defined in \eqref{Sum_SN}, thus producing an upper bound for the numerator in \eqref{exp_cond_S2}. We need some technical lemmata.
\begin{lemma}\label{lem:int}
Given $x,y\in\mathbb{R}$ and $\varphi \in [0,2\pi/r]$, we have
\[
\prod_{i=0}^{r-1}\left(x^2+y^2-2xy\cos\left(\varphi+\frac{2\pi i}{r}\right)\right)=x^{2r}+y^{2r}-2x^ry^r\cos(r\varphi).
\]
\end{lemma}
\begin{proof}
See \cite[eq. 1.394]{GR15}.
\end{proof}
Next we are going to give an exact expression for the quantity
\begin{equation}\label{eq:1}
\Theta=\prod_{i=0}^{r-1}|p-q_i|^2,
\end{equation}
where $p=(\sqrt{1-c^2},0,c)$ is a spherical point and $q_0,\ldots, q_{r-1}$ are $r$ equidistributed points in the parallel of height $h$, that is
\[
q_i=\left(\sqrt{1-h^2}\cos\left(\varphi+\frac{2\pi i}{r}\right),\sqrt{1-h^2}\sin\left(\varphi+\frac{2\pi i}{r}\right),h\right),
\]
with $\varphi\in[0,2\pi/r]$ any phase representing that the points can be in any position.

\begin{lemma}\label{lem:xey}
We have $\Theta=x^{2r}+y^{2r}-2x^ry^r\cos(r\varphi)$, where
\begin{align*}
x=&\sqrt{1-c}\sqrt{1+h},\\
y=&\sqrt{1+c}\sqrt{1-h}.
\end{align*}
In particular, the minimum and maximum of $\Theta$ are reached respectively in $\varphi=0$ and $\varphi=\pi/r$ and we get
\[
 |x^r-y^r|^2 \leqslant \Theta\leqslant |x^r+y^r|^2.
\]
\end{lemma}
\begin{proof}
We write
\begin{align*}
\Theta=&\prod_{i=0}^{r-1}|p-q_i|^2\\
=&\prod_{i=0}^{r-1}(2-2\langle p,q_i\rangle)\\
=&\prod_{i=0}^{r-1}\left(2-2\sqrt{1-h^2}\sqrt{1-c^2}\cos\left(\varphi+\frac{2\pi i}{r}\right)-2hc\right).
\end{align*}
Taking $x$ and $y$ as in the statement above and applying Lema \ref{lem:int} we deduce that
\begin{align*}
\Theta=&x^{2r}+y^{2r}-2x^ry^r\cos(r\varphi),
\end{align*}
that is the first assertion of Lemma. The second one is direct: as the maximum and the minimum value of the cosine are $\pm1$ it is enough to notice that
\[
x^{2r}+y^{2r}\pm2x^ry^r=|x^r\pm y^r|^2.
\]
\end{proof}


\begin{prop}\label{ImpCond_cota_num}
Let $p\in \mathbb{S}$ and let $S_N$ be as in \eqref{Sum_SN}. Denote by $p_i \in \mathcal{P}_N$ the points in our collection $\mathcal{P}_N$. Then, for $M\geqslant 5$,
\begin{equation*}
\log\displaystyle\prod_{k=1}^N |p-p_k| \leqslant S_N + \log 2 + \frac12.
\end{equation*}
\end{prop}

\begin{proof}
Without loss of generality, we take $p=(\sqrt{1-c^2},0,c)$ belonging to the band $B_{\ell}$, with $1 \leqslant \ell \leqslant M$. Let us denote by $q_{j,0},\ldots,q_{j,r_j-1}$ the $r_j$ equidistributed points in the $j$--th parallel.
From Lemma \ref{lem:xey}, we have
\begin{equation}\label{cota_prod_s}
\log\prod_{k=1}^N| p-p_k|=\log\prod_{j=1}^{2M-1}\prod_{i=0}^{r_j-1}| p-q_{j,i}|\leqslant \sum_{j=1}^{2M-1}\log| x_{h_j,c}^{r_j}+y_{h_j,c}^{r_j}|,
\end{equation}
where
\begin{align*}
x_{h_j,c}=&\sqrt{1-c}\sqrt{1+h_j},\\
y_{h_j,c}=&\sqrt{1+c}\sqrt{1-h_j}.
\end{align*}
Note that the bound obtained in \eqref{cota_prod_s} can be rewritten as
\begin{align}
\nonumber \sum_{j=1}^{2M-1}\log & | x_{h_j,c}^{r_j}+y_{h_j,c}^{r_j}| \\
& = \displaystyle\sum_{j=1}^{\ell-1} r_j\log x_{h_j,c} + \displaystyle\sum_{j=\ell+1}^{2M-1}r_j\log y_{h_j,c}
+\log| x_{h_{\ell},c}^{r_{\ell}}+y_{h_{\ell},c}^{r_{\ell}}| \label{cota_ProPP}\\
\nonumber & \hspace*{0.3cm} + \sum_{j=1}^{\ell-1}\log| 1+y_{h_j,c}^{r_j}/x_{h_j,c}^{r_j}|+\sum_{j=\ell+1}^{2M-1}\log|1+x_{h_j,c}^{r_j}/y_{h_j,c}^{r_j}|.
\end{align}
We know that $c\in[H_{\ell},H_{\ell-1}]$. Then, for $1\leqslant j\leqslant \ell-1$ we have $h_j\geqslant c$ and hence $x_{h_j,c}\geqslant y_{h_j,c}$. Reciprocally, for $\ell+1\leqslant
 j\leqslant 2M-1$ we get $x_{h_j,c} \leqslant y_{h_j,c}$. Moreover, if $c\in[H_{\ell},h_{\ell}]$ then $h_{\ell} \geqslant c$ and $x_{h_{\ell},c}\geqslant y_{h_{\ell},c}$, so
$$
\log |x_{h_{\ell},c}^{r_{\ell}}+y_{h_{\ell},c}^{r_{\ell}}| = r_{\ell}\log x_{h_{\ell},c} + \log |1+y_{h_{\ell},c}^{r_{\ell}}/x_{h_{\ell},c}^{r_{\ell}}| \leqslant r_{\ell}\log x_{h_{\ell},c} +\log 2,
$$
and from \eqref{Valor_Itildej} we deduce that
$$
\eqref{cota_ProPP} \leqslant \log 2 + \displaystyle\sum_{j=1}^{\ell} r_j\log x_{h_j,c} +\displaystyle\sum_{j=\ell+1}^{2M-1}r_j\log y_{h_j,c}=\log 2 +\sum_{j=1}^{2M-1}r_j \Tilde{I}_j (p).
$$
It is easy to check that this inequality also holds for $c\in(h_{\ell},H_{\ell-1}]$. 
In any case, we obtain
\begin{align*}
\sum_{j=1}^{2M-1}\log| x_{h_j,c}^{r_j}+y_{h_j,c}^{r_j}| &\leqslant \log 2 + \sum_{j=1}^{2M-1}r_j\Tilde{I}_j (p)\\
&\hspace*{0.25cm}+ \sum_{j=1}^{\ell-1}\log| 1+y_{h_j,c}^{r_j}/x_{h_j,c}^{r_j}|+\sum_{j=\ell+1}^{2M-1}\log|1+x_{h_j,c}^{r_j}/y_{h_j,c}^{r_j}|.
\end{align*}
By \eqref{Sum_SN} and using $\log(1+\alpha)\leqslant\alpha$ with $\alpha>0$, we have
\begin{align*}
\sum_{j=1}^{2M-1}\log| x_{h_j,c}^{r_j}+y_{h_j,c}^{r_j}|\leqslant & \,\,\log 2 + S_N\\
&+\sum_{j=1}^{\ell-1}\left(\frac{(1+c)(1-h_j)}{(1-c)(1+h_j)}\right)^{r_j/2}+\sum_{j=\ell+1}^{2M-1}\left(\frac{(1-c)(1+h_j)}{(1+c)(1-h_j)}\right)^{r_j/2},
\end{align*}
and bounding $H_{\ell}\leqslant c\leqslant H_{\ell-1}$ it is easy to see that
\begin{multline}
\log\prod_{k=1}^N| p-p_k| \leqslant  \log 2+S_N\\
+\sum_{j=1}^{\ell-1}\left(\frac{(1+H_{\ell-1})(1-h_j)}{(1-H_{\ell-1})(1+h_j)}\right)^{r_j/2}+
\sum_{j=\ell+1}^{2M-1}\left(\frac{(1-H_{\ell})(1+h_j)}{(1+H_{\ell})(1-h_j)}\right)^{r_j/2} \label{acot_sum}.
\end{multline}
Next, we are going to bound the two sums in \eqref{acot_sum}.
Notice that for $\ell=1$, the first sum vanishes. We have
\begin{align*}
\eqref{acot_sum} &= \sum_{j=1}^{\ell-1}\left(\frac{(1+H_{\ell-1})(1-h_j)}{(1-H_{\ell-1})(1+h_j)}\right)^{r_j/2}+\sum_{j=\ell+1}^{M}\left(\frac{(1-H_{\ell})(1+h_j)}{(1+H_{\ell})(1-h_j)}\right)^{r_j/2}\\
& \hspace*{0.5cm}
+\sum_{j=1}^{M-1}\left(\frac{(1-H_{\ell})(1-h_j)}{(1+H_{\ell})(1+h_j)}\right)^{r_j/2}\\
& = \sum_{j=1}^{\ell-1}\left(\frac{(2M^2-\ell(\ell-1))j^2}{\ell(\ell-1)(2M^2-j^2)}\right)^{2j}+\sum_{j=\ell+1}^{M}\left(\frac{\ell(\ell+1)(2M^2-j^2)}{(2M^2-\ell(\ell+1))j^2}\right)^{2j}\\
&\hspace*{0.5cm}+\sum_{j=1}^{M-1}\left(\frac{\ell(\ell+1)j^2}{(2M^2-\ell(\ell+1))(2M^2-j^2)}\right)^{2j}\\
& \leqslant \displaystyle\sum_{j=1}^{\ell-1} \left(\frac{j^2}{\ell(\ell-1)}\right)^{2j}+\displaystyle\sum_{j=\ell+1}^{M} \left(\frac{\ell(\ell+1)}{j^2}\right)^{2j}+\frac{3}{2}\displaystyle\sum_{j=1}^{M-1} \left(\frac{j}{M}\right)^{4j}\\
& \leqslant \frac12,
\end{align*}
since by Lemmas \ref{L_acot_sum_alpha} and \ref{L_acot_sum_2alpha} one can deduce that
\begin{align*}
\displaystyle\sum_{j=1}^{\ell-1} \left(\frac{j^2}{\ell(\ell-1)}\right)^{2j} &= \displaystyle\sum_{j=1}^{\ell-2} \left(\frac{j^2}{\ell(\ell-1)}\right)^{2j} + \left(\frac{\ell-1}{\ell}\right)^{2(\ell-1)}
\leqslant \frac{1}{4},
\\
\displaystyle\sum_{j=\ell+1}^{M} \left(\frac{\ell(\ell+1)}{j^2}\right)^{2j} &= \displaystyle\sum_{j=\ell+2}^{M} \left(\frac{\ell(\ell+1)}{j^2}\right)^{2j} + \left(\frac{\ell}{\ell+1}\right)^{2(\ell+1)}\\&\leqslant \displaystyle\sum_{j=\ell+2}^{M} \left(\frac{\ell+1}{j}\right)^{4j} + e^{-2} \leqslant \frac{1}{6},\\
\\
\frac{3}{2}\displaystyle\sum_{j=1}^{M-1} \left(\frac{j}{M}\right)^{4j}  &\leqslant \frac{1}{20}.
\end{align*}
\end{proof}
The final outcome of this section will be used in the proof of our main theorem:
\begin{corollary}\label{cor:C}
If $p\in B_\ell$ with $\ell\leqslant M$ and $M\geqslant 5$, then
\[
\prod_{k=1}^N |p-p_k|\leqslant 2e^{-\kappa N}\left(\frac{M}{\ell}\right)^{1/3}e^{3/4}\left(\frac{e}{2}\right)^{2/\ell}.
\]
\end{corollary}
\begin{proof}
  Immediate from  propositions  \ref{ImpCond_cota_num} and \ref{Cociente_Integral_cond_previo}.
\end{proof}


\section{The denominator of the condition number formula}\label{sec:denominador}

The results obtained in the previous section give us an upper bound for the numerator of \eqref{exp_cond_S2}.
Now, for any fixed $i=1,\ldots,N$, we need a lower bound for the denominator.

\begin{lemma}\label{prod_puntos_circunf_1}
Let $r$ be fixed and let $p_0,\ldots,p_{r-1}$ be $r$ equidistributed points on the unit circumference. Then
\begin{equation*}
\log \prod_{k=1}^{r-1} |p_k-p_0|=\log r.
\end{equation*}
\end{lemma}
\begin{proof}
This is a basic exercise in complex number theory: rotate the points in such a way that $p_{k}=e^{2i\pi k/r}$. Since the polynomial $z^r-1$ has the $r$ roots of unity as zeros, we have $\prod_{k=0}^{r-1}(z-p_{k})=z^r-1=(z-1)(1+z+\ldots+z^{r-1})$. Removing the factor $z-1$ from both terms and substituting $z=1$ we get the result.
\end{proof}
The following is an inmediate consequence:
\begin{corollary}\label{prod_puntos_circunf_r}
Let $r$ be fixed and let $p_0,\ldots,p_{r-1}$ be $r$ equidistributed points in a circumference of radius $s$. Then
\begin{equation*}
\log \prod_{k=1}^{r-1} |p_k-p_0|=\log r + (r-1)\log s.
\end{equation*}
\end{corollary}


\begin{prop}\label{ImpCond_cota_denom}
Let $p \in \mathcal{P}_N$ be fixed and let $S_N=S_N(p)$ be as in \eqref{Sum_SN}. For $M\geqslant 5$, the following inequality holds
\begin{equation*}
\log \displaystyle\prod_{\substack{p_i\in\mathcal{P}_N\\p_i\neq p}}^N |p_i-p| \geqslant S_N + \log (2\sqrt{2}M) -\frac18.
\end{equation*}
\end{prop}

\begin{proof}
We may assume, without loss of generality, that $p=q_{\ell,0}=(\sqrt{1-h_{\ell}^2},0,h_{\ell})$ is a point belonging to our set $\mathcal{P}_N$ located in the $\ell$--th parallel, with $1 \leqslant \ell \leqslant M$. As before, we denote by $q_{j,0},\ldots,q_{j,r_j-1}$ the $r_j$ equidistributed points in the $j$--th parallel. We write
\begin{align}
\nonumber\log \displaystyle\prod_{\substack{p_i\in\mathcal{P}_N\\p_i\neq p}}^N |p_i-p| =& \log \prod_{i=1}^{r_{\ell}-1}|q_{{\ell},0}-q_{\ell,i}|+
\nonumber \log\prod_{\substack{j=1 \\ j\neq \ell}}^{2M-1}\prod_{i=0}^{r_j-1}|q_{\ell,0}-q_{j,i}|\\
\geqslant & \log 2\sqrt{2}M + r_{\ell}\log x_{h_{\ell},h_{\ell}} + \sum_{\substack{j=1 \\ j \neq \ell}}^{2M-1}\log| x_{h_j,h_{\ell}}^{r_j}-y_{h_j,h_{\ell}}^{r_j}|, \label{cota_den_cond_1}
\end{align}
where we have used on the one hand that, by Lemma \ref{lem:xey},
$$
\log\prod_{\substack{j=1 \\ j\neq \ell}}^{2M-1}\prod_{i=0}^{r_j-1}|q_{\ell,0}-q_{j,i}| \geqslant \sum_{\substack{j=1 \\ j \neq \ell}}^{2M-1}\log| x_{h_j,h_{\ell}}^{r_j}-y_{h_j,h_{\ell}}^{r_j}|,
$$
with
\begin{align*}
x_{h_j,h_{\ell}}=&\sqrt{1-h_{\ell}}\sqrt{1+h_j},\\
y_{h_j,h_{\ell}}=&\sqrt{1+h_{\ell}}\sqrt{1-h_j},
\end{align*}
and on the other hand that, from Corollary \ref{prod_puntos_circunf_r}
\begin{align*}
\log \prod_{i=1}^{r_{\ell}-1}|q_{\ell,0}-q_{\ell,i}| =& \log r_{\ell} + (r_{\ell}-1)\log \sqrt{1-h_{\ell}^2}\\
=& \log \frac{r_{\ell}}{\sqrt{1-h_{\ell}^2}}+r_{\ell}\log x_{h_{\ell},h_{\ell}}\\
\geqslant & \log (2\sqrt{2}M) + r_{\ell}\log x_{h_{\ell},h_{\ell}}.
\end{align*}
For $1 \leqslant j \leqslant \ell-1$, $h_j > h_{\ell}$ so $x_{h_j,h_{\ell}} > y_{h_j,h_{\ell}}$. Reciprocally, $x_{h_j,h_{\ell}} < y_{h_j,h_{\ell}}$ for $\ell \leqslant j \leqslant 2M-1$. Thus
\begin{align*}
\sum_{\substack{j=1 \\ j \neq \ell}}^{2M-1}\log| x_{h_j,h_{\ell}}^{r_j}-y_{h_j,h_{\ell}}^{r_j}|
=& \sum_{j=1}^{\ell-1}r_j\log x_{h_j,h_{\ell}}+ \sum_{j=\ell+1}^{2M-1}r_j\log y_{h_j,h_{\ell}}\\
& + \sum_{j=1}^{\ell-1}\log| 1-y_{h_j,h_{\ell}}^{r_j}/x_{h_j,h_{\ell}}^{r_j}|
+ \sum_{j=\ell+1}^{2M-1}\log| 1-x_{h_j,h_{\ell}}^{r_j}/y_{h_j,h_{\ell}}^{r_j}|,
\end{align*}
and substituting into \eqref{cota_den_cond_1} we obtain
\begin{align}
\nonumber \log \displaystyle\prod_{\substack{p_i\in\mathcal{P}_N\\p_i\neq p}}^N |p_i-p|  \geqslant & \log (2\sqrt{2}M) + S_N + T_{\ell,j},\\
T_{\ell,j}& = \sum_{j=1}^{\ell-1}\log| 1-y_{h_j,h_{\ell}}^{r_j}/x_{h_j,h_{\ell}}^{r_j}|
+ \sum_{j=\ell+1}^{2M-1}\log| 1-x_{h_j,h_{\ell}}^{r_j}/y_{h_j,h_{\ell}}^{r_j}|, \label{log_cotden}
\end{align}
since by \eqref{Valor_Itildej} and \eqref{Sum_SN}
$$
\sum_{j=1}^{\ell}r_j\log x_{h_j,h_{\ell}}+ \sum_{j=\ell+1}^{2M-1}r_j\log y_{h_j,h_{\ell}} = \sum_{j=1}^{2M-1}r_j \Tilde{I}_j (q_{\ell,0}) = S_N(p).
$$
Next, we use
\begin{equation}\label{des_log_m}
\log(1-\alpha)\geqslant -\frac{16}{15}\alpha, \quad \alpha \in [0,1/16],
\end{equation}
to estimate \eqref{log_cotden} (note that for $\ell=1$, the first sum vanishes).
It is not difficult to check that we can use \eqref{des_log_m} to get the following bounds for \eqref{log_cotden}.

\begin{align*}
-T_{\ell,j} \leqslant &
\frac{16}{15}\displaystyle\sum_{j=1}^{\ell-1} \left(\frac{y_{h_{j},h_{\ell}}}{x_{h_{j},h_{\ell}}}\right)^{r_j} +\frac{16}{15}\displaystyle\sum_{j=\ell+1}^M\left(\frac{x_{h_{j},h_{\ell}}}{y_{h_{j},h_{\ell}}}\right)^{r_j}+\frac{16}{15}\displaystyle\sum_{j=M+1}^{2M-1}\left(\frac{x_{h_{j},h_{\ell}}}{y_{h_{j},h_{\ell}}}\right)^{r_j}\\
= & \frac{16}{15}\displaystyle\sum_{j=1}^{\ell-1}\left(\frac{(2M^2-\ell^2)j^2}{\ell^2(2M^2-j^2)}\right)^{2j} + \frac{16}{15}\displaystyle\sum_{j=\ell+1}^M
\left(\frac{\ell^2(2M^2-j^2)}{(2M^2-\ell^2)j^2} \right)^{2j}\\
& \hspace{0.5cm}+\frac{16}{15}\displaystyle\sum_{j=1}^{M-1} \left(\frac{\ell^2j^2}{(2M^2-\ell^2)(2M^2-j^2)}\right)^{2j}\\
\leqslant & \frac{16}{15}\displaystyle\sum_{j=1}^{\ell-1} \left(\frac{j}{\ell}\right)^{4j}+ \frac{16}{15}\displaystyle\sum_{j=\ell+1}^M \left(\frac{\ell}{j}\right)^{4j}+ \frac{16}{15}\displaystyle\sum_{j=1}^{M-1} \left(\frac{j}{M}\right)^{4j}\\
\leqslant & \frac1{8},
\end{align*}
where we have used lemmas \ref{L_acot_sum_alpha} and \ref{L_acot_sum_2alpha}. The proposition follows.
\end{proof}

We will use the following easy consequence in the last section:
\begin{corollary}\label{cor:D}
Let $p$ be any point of $\mathcal P_N$. Then
\[
\displaystyle\prod_{\substack{p_i\in\mathcal{P}_N\\p_i\neq p}}^N |p_i-p|  \geqslant \sqrt{2N}e^{-\kappa N}e^{-9/8}.
\]
\end{corollary}
\begin{proof}
  Immediate from  propositions  \ref{ImpCond_cota_denom} and \ref{Cociente_Integral_cond_previo}.
\end{proof}


\section{Proof of the main results}\label{sec:final}

If $M\leqslant 4$ our proof is computer assisted: we construct the polynomial $P_N$, (that has rational coefficients) and the point set (whose points are algebraic and can thus be represented exactly in a computer algebra package), which allows us to compute exactly $\mu_{norm}$ from \eqref{eq:mu}, showing that its actually upper bounded by $N=4M^2$.
For $M\geqslant 5$, from \eqref{exp_cond_S2} and the symmetry of the construction, we have
\begin{align*}
\mu_{\text{norm}}(P)&=\frac{1}{2}\sqrt{N(N+1)}\max_{1 \leqslant i \leqslant N}\frac{\left(\int_{\mathbb{S}} \prod_{j=1}^N |p- p_j|^2 d\sigma(p)\right)^{1/2}}{\prod_{j\neq i}| p_i- p_j|}\\
&\stackrel{\text{Cor. \ref{cor:D}}}{\leqslant}\frac{1}{2}\sqrt{N(N+1)}\frac{\left(\sum_{\ell=1}^{2M-1}\int_{B_\ell} \prod_{j=1}^N |p- p_j|^2 d\sigma(p)\right)^{1/2}}{\sqrt{2N}e^{-\kappa N}e^{-9/8}}\\
&=\frac{1}{2}\sqrt{N(N+1)}\frac{\left(\int_{B_M} \prod_{j=1}^N |p- p_j|^2 d\sigma(p)+2\sum_{\ell=1}^{M-1}\int_{B_\ell} \prod_{j=1}^N |p- p_j|^2 d\sigma(p)\right)^{1/2}}{\sqrt{2N}e^{-\kappa N}e^{-9/8}}
\end{align*}
Using Corollary \ref{cor:C} and recalling that the relative area of $B_\ell$ is $r_\ell/N=\ell/M^2$, the term inside the parenthesis is bounded above  by
\begin{align*}
4e^{-2\kappa N}e^{3/2}\left(\frac{1}{M}\left(\frac{e}{2}\right)^{4/M}+\frac{2}{M^{4/3}}\sum_{\ell=1}^{M-1}\ell^{1/3}\left(\frac{e}{2}\right)^{4/\ell}\right),\\
\stackrel{\text{Lemma \ref{lem:sum3}}}{\leqslant} 4e^{-2\kappa N}e^{3/2}\left(\frac{1}{M}\left(\frac{e}{2}\right)^{4/M}+\frac32+\frac{24(1-\log 2)}{M}+\frac{6}{M^{4/3}}\right).
\end{align*}

We then have proved:
\begin{align*}
\mu_{\text{norm}}(P)&\leqslant \sqrt{\frac{N+1}2}e^{3/4+9/8}\left(\frac{1}{M}\left(\frac{e}{2}\right)^{4/M}+\frac32+\frac{24(1-\log 2)}{M}+\frac{6}{M^{4/3}}\right)^{1/2}.
\end{align*}
This proves our Theorem \ref{th:mainintro}: the term inside the parenthesis decreases with $M$ and conclude after some arithmetic:
\[
\mu_{\text{norm}}(P)\leqslant\frac{19}{2}\sqrt{N+1},
\]
which is less than $N$ for $M\geqslant 5$. Moreover, we also get a proof of Corollary \ref{cor:limsinf} since in the limit $M\to\infty$ we have
\[
\mu_{\text{norm}}(P)\leqslant \frac{\sqrt{3}}{2}e^{3/4+9/8}\sqrt{N+1}.
\]


\appendix

\section{Some technical results used in the proofs}

The following formula is a consequence of the change of variables theorem, sending a point $(a,b,c)\in\mathbb S$ to the cylinder $((a^2+b^2)^{-1/2}a,(a^2+b^2)^{-1/2}b,c)$:

\begin{lemma}\label{int_S2}
Let $f$ be integrable on $[-1,1]$. Then,
$$
\int_{\mathbb{S}}f(\langle x, (0,0,1)\rangle)d\sigma(x) = \frac12\int_{-1}^1 f(t)dt.
$$
\end{lemma}
We have also used the following elementary estimate.
\begin{lemma}\label{lem:sumalog}
Let $M\geqslant 2$ and $1\leqslant \ell\leqslant M-1$. Then,
\[
\log\frac{M+1}{\ell+1}\leqslant \sum_{j=\ell+1}^M\frac1j\leqslant \log\frac{M}{\ell}
\]
\end{lemma}
\begin{proof}
  This follows from the comparison of the sum and the associated integral: $\int_{\ell+1}^{M+1}\frac{1}{x}\,dx$ for the lower bound and $\int_{\ell+1}^{M+1}\frac{1}{x-1}\,dx$ for the upper bound.
\end{proof}
\section{Some discrete sums}
In this section we prove some technical, elementary estimates that have been used in the proofs of the paper.


\begin{lemma}\label{L_acot_sum_alpha}
For $M\geqslant 2$, let
\begin{equation*}\label{serie_ppa}
R(M)=\displaystyle\sum_{j=1}^{M-1} \left(\frac{j}{M}\right)^{4j}.
\end{equation*}
Then, $R(M)\leqslant 1/16$ for all $M\geqslant 2$, and $R(M)\leqslant 1/30$ for $M\geqslant 5$.
\end{lemma}

\begin{proof}

It is easy to check with some rational computations that $R(5)\leqslant \cdots\leqslant R(3)\leqslant R(2)=1/16$ and also that $R(5)\leqslant 1/30$. We finish the proof by showing that $R(M)\leqslant 1/30$ for $M\geqslant 6$. Indeed, note that the sum in the lemma is $M$ times a composite midpoint rule of the convex function $x^{4Mx}$ and we thus have:

\begin{equation*}
R(M) \leqslant M\int_{\frac{3}{2M}}^{1-\frac{7}{2M}} x^{4Mx}dx+\displaystyle\sum_{j=M-3}^{M-1} \left(\frac{j}{M}\right)^{4j}+\frac{1}{M^4},
\end{equation*}
and we bound each of the previous terms. On the one hand, it is easy to see that
\begin{equation*}
\displaystyle\sum_{j=M-3}^{M-1} \left(\frac{j}{M}\right)^{4j}+\frac{1}{M^4}
\end{equation*}
is decreasing in $M$ and thus it is bounded above by $7/240$ for $M\geqslant 6$. On the other hand, the function $x^{4Mx}$ is decreasing in $x$ in the interval $[3/2M,e^{-1}]$ and is increasing for $[e^{-1},1-7/2M]$, so
\begin{equation*}
M\int_{\frac{3}{2M}}^{1-\frac{7}{2M}} x^{4Mx}dx \leqslant \left(\frac{3}{2}\right)^6\int_{\frac{3}{2M}}^{\frac{1}{e}}\frac{1}{M^5}dx+M\int_{\frac{1}{e}}^{1-\frac{7}{2M}} x^{4M/e}dx.
\end{equation*}
Solving the immediate integrals above we get an increasing function in $M$ whose upper bound is $1/240$.
The sum of the two upper bounds obtained is less than or equal to $1/30$ as claimed.
\end{proof}


\begin{lemma}\label{L_acot_sum_2alpha}
For  $1 \leqslant \ell \leqslant M-2$, we have
\begin{equation*}
\displaystyle\sum_{j=\ell+2}^{M}\left(\frac{\ell+1}{j}\right)^{4j} \leqslant \frac{1}{e^4-1}.
\end{equation*}
\end{lemma}

\begin{proof}
Note that for all $j\geqslant \ell+1$
\[
\left(\frac{\ell+1}{j}\right)^{4j}=\left(1-\frac{j-\ell-1}{j}\right)^{4j}\leqslant e^{-4(j-\ell-1)},
\]
which yields the following upper bound for the sum in the lemma:
\begin{equation*}
\displaystyle\sum_{j=\ell+2}^{M}e^{-4(j-\ell-1)}\leqslant \sum_{k=1}^\infty e^{-4k}=\frac{1}{e^4-1}.
\end{equation*}

\end{proof}

The following sum has appeared in the proof of our main theorem:
\begin{lemma}\label{lem:sum3}
The following inequality holds:
\[
\sum_{\ell=1}^{M-1}\ell^{1/3}\left(\frac{e}{2}\right)^{4/\ell}\leqslant \frac34M^{4/3}+12(1-\log 2)M^{1/3}+3.
\]
\end{lemma}
\begin{proof}
  Write
  \[
  \left(\frac{e}{2}\right)^{4/\ell}=e^{4(1-\log 2)/\ell}=1+\frac{4(1-\log 2)}{\ell}+\sum_{k=2}^\infty \frac{(4(1-\log 2))^k}{k!\ell^k},
  \]
  and interchange the summation symbols to get the following expression, equivalent to the sum in the lemma:
  \[
  \sum_{\ell=1}^{M-1}\ell^{1/3}+4(1-\log 2)\sum_{\ell=1}^{M-1}\ell^{-2/3}+\sum_{k=2}^\infty \frac{(4(1-\log 2))^k}{k!}\sum_{\ell=1}^{M-1}\ell^{1/3-k}.
  \]
  We can upper bound the first of these three sums by $\int_1^Mx^{1/3}\,dx\leqslant 3M^{4/3}/4$, and the second one by $1+\int_1^{M-1}x^{-2/3}dx\leqslant 3M^{1/3}$. The third term is at most
  \begin{multline*}
  \sum_{k=2}^\infty \frac{(4(1-\log 2))^k}{k!}\left(1+\int_1^{M-1}x^{1/3-k}\,dx\right)\leqslant \\
  \sum_{k=2}^\infty \frac{(4(1-\log 2))^k}{k!}\left(1+\frac{3}{3k-4}\right)\\
  \leqslant \frac52\left(\frac{e^4}{2^4}-1-4(1-\log 2)\right)\leqslant 3,
  \end{multline*}
  where we have used some arithmetic for the last step.
\end{proof}

\section*{Acknowledgements}

The authors were partially supported by Ministerio de Econom\'ia y Competitividad, Gobierno de Espa\~na, through grants MTM2017-83816-P, MTM2017-90682-REDT, and by the Banco de Santander and Universidad de Cantabria grant \, 21.SI01.64658.


\end{document}